\documentclass{article}

\usepackage{amsmath}
\usepackage{amsfonts}
\usepackage{amssymb}
\usepackage{amsthm}
\usepackage{enumerate}
\usepackage{bbm}
\usepackage{hyperref}
\usepackage{verbatim}

\newtheorem{theorem}{Theorem}[section]

\newtheorem{claim}[theorem]{Claim}
\newtheorem{definition}[theorem]{Definition}
\newtheorem{corollary}[theorem]{Corollary}

\newtheorem{conjecture}[theorem]{Conjecture}

\newcommand{\E}{{\rm I\kern-.3em E}}

\begin{document}
\title{On the list recoverability of randomly punctured codes}
\author{
Ben Lund \thanks{Department of Mathematics,  Princeton University. {\tt lund.ben@gmail.com}. Research supported by NSF grant 
DMS-1802787.} 
\and
Aditya Potukuchi \thanks{Department of Computer Science,  Rutgers University. {\tt aditya.potukuchi@cs.rutgers.edu}. Research supported in part by Noga Ron-Zewi's BSF-NSF grant CCF-1814629 and 2017732, by NSF grant  CCF-1350572 and the Simons Collaboration on Algorithms and Geometry}
}

\maketitle

\begin{abstract}
We show that a random puncturing of a code with good distance is list recoverable beyond the Johnson bound.
In particular, this implies that there are Reed-Solomon codes that are list recoverable beyond the Johnson bound.
It was previously known that there are Reed-Solomon codes that do not have this property. 
As an immediate corollary to our main theorem, we obtain better degree bounds on unbalanced expanders that come from Reed-Solomon codes.
\end{abstract}

\section{Introduction}

List recoverable codes were defined by Guruswami and Indyk \cite{guruswami2001expander} in their study of list decodable codes.
Here, we study list recoverable codes in their own right, showing that random puncturings of codes over a sufficiently large alphabet are list recoverable.
Our result is analogous to earlier work by Rudra and Wooters \cite{RW14, RW15} on the list decodability of randomly punctured codes. 

We use $q$ to denote the alphabet size, and $n$ to denote the block length of an arbitrary code $\mathcal{C} \subset [q]^n$.
Given two codewords $c_1, c_2 \in [q]^n$, denote the Hamming distance between $c_1$ and $c_2$ by $\Delta(c_1, c_2)$.
Denote the minimum distance between a codeword $c \in [q]^n$ and a set $\mathcal{L} \subseteq [q]^n$ by $\Delta(c, \mathcal{L})$.

\begin{definition}[List recoverability]
Let $q,n,\ell, L$ be positive integers, and let $0 \leq \rho < 1$ be a real number.
A code $\mathcal{C} \subset [q]^n$ is $(\rho, \ell, L)$ list recoverable if, for every collection of sets $\{A_i \subseteq [q]\}_{i \in [n]}$ with $|A_i| \leq \ell$ for each $i$, we have
\[
\left|\{c \in \mathcal{C}~|~\Delta(c,A_1 \times \cdots \times A_n) \leq \rho n \}\right| \leq L.
\]
\end{definition}

In the above definition, $\ell$ is called the \emph{input list size}, and $L$ is called the \emph{output list size} from which the code can be recovered.
The case $\rho = 0$ is already interesting, and is called {\em zero-error} list recoverability.
We say that a code $\mathcal{C}$ is $(\ell,L)$ zero-error list recoverable if it is $(0, \ell, L)$ list recoverable.

A {\em puncturing} of a code $\mathcal{C} \subset [q]^n$ to a set $S \subset [n]$ is the code $\mathcal{C}_S \subset [q]^S$ defined by $\mathcal{C}_S[i] = \mathcal{C}[i]$ for each $i \in S$. 
A punctured code will typically have higher rate, but lower distance, than the unpunctured version. 
Our main result is that every code over a large enough alphabet $[q]$ can be punctured to a code of rate $R > q^{-1/2}$ while being list recoverable with input and output list sizes roughly $R^{-2}$.

\begin{theorem}
\label{unbalanced:thm:simple}
For any real $\alpha$ with $0 < \alpha \leq 1$, every code with distance at least $n(1-q^{-1}-\epsilon^{2})$ can be punctured to rate $\Omega\left(\frac{\epsilon}{\log q}\right)$ so that it is $(\rho, \ell, \ell(1 + \alpha))$-list recoverable, provided that the following inequalities are satisfied:
 \begin{itemize}
\item $n$ and $q$ are sufficiently large,
\item $0 \leq \rho < 1 - (1+\alpha)^{-1/2}$,
\item $q^{-1/2} < \epsilon < \min(c, 2^{-1} \gamma \sigma)$,
\item $\ell \leq \epsilon^{-2}\sigma^2\gamma$,
\end{itemize}
where $c > 0$ is a constant, $\gamma = (1+\alpha)(1 - \rho)^2 - 1$, and $\sigma=(1-\rho)(2-\rho)^{-1}$.
\end{theorem}

In fact, we show a random puncturing is list recoverable with the same list size with high probability; see Theorem \ref{unbalanced:thm:main} for a precise statement. 

It is helpful to consider the case $\rho=0$ and $\alpha = 1$.
Then we obtain a $(\ell, 2\ell)$ zero-error list-recoverable code.
If we start with a code having distance nearly as large as possible ({\em i.e.} take $\epsilon = \Theta(q^{-1/2})$), then we obtain input and output list sizes of $\Omega(q)$ and rate $\Omega(q^{-1/2}\log^{-1}q)$.
In this way, we obtain a punctured code of rate $\Omega(\ell^{-1/2}\log^{-1} \ell)$.
The main point is that this is better than the codes guaranteed by the Johnson bound (discussed in more detail below), which gives a code of rate $\Omega(\ell^{-1})$.
A completely random code, however, can be $(\ell, 2\ell)$ zero-error list-recoverable with rate $\Omega(1)$. 

One nice feature Theorem \ref{unbalanced:thm:simple} is that it can yield an input list size as large as $\Omega(q)$. A second nice feature is that the theorem yields a tight relationship between the input and output list sizes; many results on list recovery give an output list that is much larger than the input list.
Finally we mention that the proof of Theorem \ref{unbalanced:thm:simple} is relatively simple and elementary; we draw a slightly more explicit comparison to the earlier work of Rudra and Wooters~\cite{RW14,RW15} below.

\section{Background}

In this section, we give a brief discussion about the current state of the literature. Theorem \ref{unbalanced:thm:simple} is analogous to a theorem of Rudra and Wooters~\cite{RW14,RW15} on the {\em list decodability} of punctured codes over large alphabets.
A code $\mathcal{C} \subset [q]^n$ is $(\rho, \ell)$-list decodable if for each $x \in [q]^n$, there are at most $\ell$ codewords of $\mathcal{C}$ that differ from $x$ in fewer than $\rho n$ coordinates.

\begin{theorem}[\cite{RW15}]\label{unbalanced:thm:listDecodability}
Let $\epsilon > q^{-1/2}$ be a real number, and $q,n$ be sufficiently large integers. Every code $\mathcal{C} \subset [q]^n$ with distance $n(1-q^{-1}-\epsilon^2)$ can be punctured to rate $\tilde{\Omega}\left(\frac{\epsilon}{\log q}\right)$ so that it is $(1 - O(\epsilon),O(\epsilon^{-1}))$-list decodable.
\end{theorem}

Theorems \ref{unbalanced:thm:simple} and \ref{unbalanced:thm:listDecodability} are most interesting in the case of Reed-Solomon codes.
The codewords of the degree-$d$ Reed-Solomon code over $\mathbb{F}_q$ with evaluation set $S \in \binom{[q]}{n}$ are the evaluations of all univariate polynomials of degree at most $d$ on elements of $S$.
In other words, suppose $S = \{s_1,\ldots,s_n\}$.
The degree-$d$ Reed-Solomon code on $S$ is the set
\[
\{(p(s_1),\ldots,p(s_n))~|~\operatorname{deg}(p) \leq d\}.
\]
The block length of this code is $n \leq q$.
Since two distinct polynomials of degree at most $d$ can agree on at most $d$ locations, the distance of any degree-$d$ Reed-Solomon code is at least $n-d$. 

A fundamental result, which gives a lower bound on the list decodability of a code with given distance, is the {\em Johnson bound} (see, for example Corollary $3.2$ in ~\cite{GUR06}).
\begin{theorem}[Johnson bound for list decoding]
\label{unbalanced:thm:jbdecode}
Every code $\mathcal{C} \subset [q]^n$ of minimum distance at least $n(1 - q^{-1} - \epsilon^2)$ is $((1 - q^{-1} - \epsilon), O(\epsilon^{-1}))$- list decodable. 
\end{theorem}
One of the main points of Theorem \ref{unbalanced:thm:listDecodability} is that it shows that there are Reed-Solomon codes that are list decodable beyond the Johnson bound. This is a very interesting result because variations of Reed-Solomon codes have been shown to beat the Johnson bound. On the other hand, Reed-Solomon codes are quite natural and some applications in Complexity theory specifically needed Reed-Solomon codes (and not the aforementioned variants). 

A similar result as Theorem~\ref{unbalanced:thm:jbdecode}, using a similar argument, also known as the Johnson bound, is known for list recoverability (see for example, Lemma $5.2$ in~\cite{GKORS18}). 

\begin{theorem}[Johnson bound for list recovery]
\label{unbalanced:thm:jbrecover}
Let $\mathcal{C} \subseteq [q]^n$ be a code of relative distance $1 - \epsilon$. Then $\mathcal{C}$ is $(\rho, \ell, L)$-list recoverable for any $\ell \leq \epsilon^{-1}(1 - \rho)^2$ with $L = \frac{\ell}{(1-\rho)^2 - \epsilon\ell} $.
\end{theorem}

For perspective, the goal here is to have the input list size $\ell$ as large as possible for a given $\rho$ while ensuring that the output list size $L$ is small, for example, $L = \operatorname{poly}(\ell)$ is a reasonable regime of efficiency. Theorem~\ref{unbalanced:thm:jbrecover} roughly says that every $q$-ary code of distance $1 - q^{-1}  - \epsilon$ is $(\rho, \ell, O(\ell))$-list recoverable for $\ell = O(\epsilon^{-1})$ and $\rho \leq 1 - \sqrt{2\epsilon\ell}$. Thus, a question naturally arises: \\

\textbf{Question:} \textit{Are there $q$-ary Reed-Solomon codes of distance $1 - q^{-1} - \epsilon$ which are $(\rho, \ell, L)$-list recoverable for some $\rho \in [0,1)$ and $\ell = \omega(\epsilon^{-1})$ and $L = \operatorname{poly(\ell)}$?}\\

We would like to remark again, that the case $\rho = 0$ is also interesting. A result of Guruswami and Rudra~\cite{GR06} shows that there are Reed-Solomon codes where one cannot hope to beat the Johnson bound in the case when $\rho = 0$.

\begin{theorem}[Theorem $1$ in \cite{GR06}]
\label{unbalanced:thm:GR}
Let $q = p^{m}$ where $p$ is a prime, and let $\mathcal{C}$ denote the degree-$\left(\frac{p^m-1}{p-1}\right)$ Reed-Solomon code over $\mathbb{F}_q$ with $\mathbb{F}_q$ as the evaluation set. Then there are lists $S_1,\ldots, S_q$ each of size $p$ such that 
\[
|\mathcal{C} \cap (S_1 \times \cdots \times S_q)| = q^{2^m}.
\]
\end{theorem}

In the proof of the above theorem, each list $A_i$ is given by $\mathbb{F}_p \subset \mathbb{F}_q$, and the set of codewords contained in these lists are codewords of the \emph{BCH code}. This result exploits the specific subfield structure of the input lists. To understand this result quantitatively, recall that a degree-$d$ Reed-Solomon code has relative distance $1 - \frac{1}{q} - \frac{d}{q}$. 
Setting $\ell = p-1$ and $\rho = 0$ in the Johnson bound tells us that such a code is $(p-1,O(q))$ zero-error list recoverable. 
Setting the list size as $p$ in the bound gives us nothing, and Theorem~\ref{unbalanced:thm:GR} says that the number of codewords grows superpolynomially in $q$. 

One thing to note in Theorem~\ref{unbalanced:thm:GR} is that the same lists of size $p$ also support at least $q^{2^m}$ codewords for any punctured code. Thus in some sense, the reason why  Theorem~\ref{unbalanced:thm:simple} is true is that lists of size $p-1$ do not support many more codewords in an appropriately punctured code than in the un-punctured case. This is very similar to the intuition in the results of~\cite{RW14},~\cite{RW15} for list-decodability.

Theorem~\ref{unbalanced:thm:simple} immediately gives the following corollary.

\begin{corollary}
For a prime power $q$ and small enough $\epsilon \geq q^{-1/2}$, there are Reed-Solomon codes of rate $\tilde{\Omega}\left(\frac{\epsilon}{\log q}\right)$ which are $(\epsilon^{-2}, O(\epsilon^{-2}))$ zero-error list recoverable.
\end{corollary}

Thus, one is able to go beyond the $O(\epsilon^{-1})$ size input lists. Again, one can easily check that setting $\ell = \epsilon^{-2}$ in the Johnson bound gives nothing. A point worth noting is that there is no inherent upper bound on the input list size, and the input lists can be as large as $\Omega(q)$ when $\epsilon$ is around $q^{-1/2}$. In fact, $\ell = \Omega(q)$ and $L \leq 2\ell$ is the regime for the main application of Theorem~\ref{unbalanced:thm:simple} (see Section~\ref{unbalanced:sec:motivation}). 

A natural attempt at Theorem~\ref{unbalanced:thm:simple} is to use the method from the aforementioned result of Rudra and Wootters. This method uses tools from high dimensional probability. At a \emph{very} high level, the main intuition in their argument is the following: Suppose there is a set $\Lambda$ of $L = \Omega(\epsilon^{-1})$ codewords which are at a distance at most $1 - O(\epsilon)$ from some point $x \in [q]^n$, then there is a subset $\Lambda' \subset \Lambda$ which is much smaller where the distribution of distances from $x$ is similar to that of $\Lambda$. The existence of such a $\Lambda$ is a bad event, which is witnessed by a smaller bad event (i.e., existence of $\Lambda'$) of the \emph{same type}. Thus, this requires one to union bound just over $\Lambda'$'s which are much smaller in number. Attempting to use this idea in a straightforward way to list recovery seems to be very lossy. The proof of Theorem~\ref{unbalanced:thm:simple} builds on this idea, and uses the fact that bad events can be witnessed by smaller bad events of a \emph{different} (although related) type (see Section~\ref{section:sketch} for a somewhat more detailed sketch of the proof). Surprisingly though, the execution of this idea in this case is far simpler than in the previous works, and is completely elementary. 

\subsection{A quantitative summary of rate bounds for list recovery} We summarize the above discussion into a perspective with which one may view Theorem~\ref{unbalanced:thm:simple}. Fix a $\rho \in [0,1)$ to be the fraction of errors from which we wish to list-recover $q$-ary codes where $q$ is larger than the block length. Suppose one wanted to list recover from input lists of size $\ell$ where the output list is $\operatorname{poly}(\ell)$\footnote{The proof of Theorem~\ref{unbalanced:thm:simple} relies on a ``birthday paradox'' type argument that cannot exploit the additional structure when one allows $L$ to grow as a larger function of $\ell$.}, the Johnson bound (Theorem~\ref{unbalanced:thm:jbrecover}) guarantees that Reed-Solomon codes of rate $\Omega(1/\ell)$ achieve this. Theorem~\ref{unbalanced:thm:GR} says that in general, this dependence cannot be improved, i.e., there are Reed-Solomon codes of rate $O(1/\ell)$ where the output list is superpolynomial for infinitely many $q$ and $\ell$. However, Theorem~\ref{unbalanced:thm:simple} says that most Reed-Solomon codes of rate $\Omega(1/(\ell^{1/2}\log q))$ achieve this.  

It should be worth noting that decoding from lists of size $\ell$ can be achieved by random codes of rate $\Omega(1)$, whereas random Reed-Solomon codes require that the rate is $O(1/\log \ell)$ (Theorem~\ref{unbalanced:thm:upperbound}). In fact, nothing better is known even for random linear codes. In this sense random codes are much better at list recovery than random Reed-Solomon codes.

\subsection{Unbalanced expander graphs from codes}
\label{unbalanced:sec:motivation}

The zero-error case of Theorem \ref{unbalanced:thm:simple} leads to some progress on a question of Guruswami regarding unbalanced expanders obtained from Reed-Solomon graphs. This was also the main motivation behind this theorem. 

Informally, an expander graph is a graph where every small set of vertices has a relatively large neighborhood. 
In this case, we say that all small sets \emph{expand}. 
One interesting type of expander graphs are \emph{unbalanced expanders}. 
These are bipartite graphs where one side is much larger than the other side, and we want that all the small subsets of the \emph{larger} side expand.

\begin{definition}[Unbalanced expander]
A $(k,d,\epsilon)$-regular unbalanced expander is a bipartite graph on vertex set $L \sqcup R$, $|L| \geq |R|$ where the degree of every vertex in $L$ is $d$, and for every $S \subseteq L$ such that $|S| = k$, we have that $|N(S)|\geq d|S|(1 - \epsilon)$.
\end{definition}

Note that in the above definition, $|N(S)| \leq d|S|$. 
We are typically interested in infinite families of unbalanced expanders for which $\epsilon = o(1)$, $d = o(|R|)$, and $k = \tilde{\Omega}(|R|/d)$. 

Given a code $\mathcal{C} \subseteq [q]^n$, it is natural to look at the bipartite graph, which we will denote by $G(\mathcal{C})$ where the vertex sets are $\mathcal{C} \sqcup ([n] \times [q])$. For every $c = (c_1,\ldots, c_n) \in \mathcal{C}$ the set of neighbors is $\{(1,c_1),\ldots,(n,c_n)\}$. 
This graph is especially interesting when $\mathcal{C}$ is a low-degree Reed-Solomon code evaluated at an appropriate set. 

The following is a open question in the study of pseudorandomness that is attributed to Guruswami~\cite{GUR}, (also explicitly stated in~\cite{CZ18}): Fix an integer $d$. For a subset $S \in \binom{[q]}{m}$, define $\mathcal{C}_S$ to be the degree-$d$ Reed-Solomon code with $S$ as the evaluation set, where $d$ is a constant. \\

\textbf{Question:}  \textit{What is the smallest $m$ such that when $S$ is chosen uniformly at random, $G(\mathcal{C}_S)$ is, with high probability, a $(\Omega(q),1/2)$-unbalanced expander?}  \\

There are examples of explicit constructions unbalanced expanders that come from other means~\cite{GUV09}. In fact,~\cite{GUV09} also contains a construction based on a variant of Reed-Solomon code known as \emph{folded Reed-Solomon code}. However, the above question has a very natural geometric/combinatorial ``core'' which is interesting in its own right and so far, seems to evade known techniques. 

It was probably well known that $m = \Omega(\log q)$, and we also give a proof of this (Theorem~\ref{unbalanced:thm:upperbound}) since we could not find it in the literature. But for upper bounds, it seems nothing better than the almost trivial $m = O(q)$ was known~\cite{CHE}. Since the zero-error list recoverability of $\mathcal{C}$ is equivalent to the expansion of $G(\mathcal{C})$, an immediate Corollary to Theorem~\ref{unbalanced:thm:main} gives an improved upper bound.

\begin{corollary}
Let  $q,n$ be sufficiently large integers and $\alpha \in (0,1)$, $\epsilon > q^{-1/2}$ be real numbers. For every code $\mathcal{C} \subset [q]^n$ with relative distance $1-q^{-1}-\epsilon^2$, there is a subset $S \subseteq [n]$ such that $|S| = O(\epsilon n \log q)$ such that $G(\mathcal{C}_S)$ is a $(\alpha \epsilon^{-2}, |S|, \alpha)$-unbalanced expander.
\end{corollary}

Instantiating the above theorem for degree-$d$ Reed-Solomon codes, we have $n = q$ and $\epsilon = (d/q)^{-\frac{1}{2}}$. This gives, $m = \tilde{O}(\sqrt{q})$.

\section{Proof of Theorem~\ref{unbalanced:thm:simple}}

The bulk of this section is the statement and proof of Theorem \ref{unbalanced:thm:main}.
After the proof of Theorem \ref{unbalanced:thm:main}, we show how to derive Theorem \ref{unbalanced:thm:simple} from it.

\subsection{A sketch of the proof}
\label{section:sketch}
Here, we sketch the proof when $\rho = 0$, i.e., for \emph{zero-error} list recovery. This contains most of the main ideas required for the general theorem. Let $S = \{s_1,\ldots, s_m\} \subseteq  [n]$ be a randomly chosen evaluation set. The main observation is that if there are input lists $A_1,\ldots,A_m \subseteq [q]$, such that $(A_1\times \cdots \times A_m)$ contains a large subset $\mathcal{D} \subseteq \mathcal{C}$ of codewords, then there is a small subset $\mathcal{C}'\subseteq \mathcal{D} \subseteq \mathcal{C}$ which agree on an unusually high number of coordinates.  An appropriately sized random subset of $\mathcal{D}$ does this. Thus the event that a given puncturing is bad is contained witnessed by the event that there are few codewords that agree a lot on the coordinates chosen in $S$. The number of events of the latter kind are far fewer in number, leaving us with fewer bad events to overcome for the union bound.

\subsection{Proof of Theorem~\ref{unbalanced:thm:simple}}

The calculations in the proof of Theorem \ref{unbalanced:thm:main} are all explicit, but we have not tried to optimize the constant terms.

\begin{theorem}
\label{unbalanced:thm:main}
Let $0 < \alpha < 1$ and $0 \leq \rho < 1 - (1+\alpha)^{-1/2}$ be real numbers.
Let $q,n,d,\ell,$ and $m$ be positive integers.
Let $\mathcal{C} \subset [q]^n$ be a code of distance at least $n - nq^{-1} - d$.
Denote $\gamma = (1+\alpha)(1-\rho)^2 - 1$ and $\sigma = (1-\rho)(2-\rho)^{-1}$.
Suppose that the following inequalities are satisfied:
\begin{align*}
    d &\geq n q^{-1}, \\
    4 \gamma^{-1} &\leq \ell \leq 800^{-1} \, \sigma \gamma n d^{-1}, \\
    \sigma m &\geq 1280 \sqrt{\ell \gamma^{-1}} \log |\mathcal{C}|, \\
    m &< n.
\end{align*}

Then, for $S \in \binom{[n]}{m}$ chosen uniformly at random, the probability that $\mathcal{C}_S$ is $(\rho,\ell,\ell(1 + \alpha))$-list recoverable is at least $1 - e^{-\sigma m / 64}$.
\end{theorem}

\begin{proof}
	For any $\mathcal{C}' \subseteq \mathcal{C}$, denote by $T(\mathcal{C}')$
	the set of coordinates $i \in [n]$ such that there is a pair $c_1,c_2 \in \mathcal{C}'$ with $c_1[i] = c_2[i]$. 
	
	The basic outline of the proof is to first show that, for any $S$ such that $\mathcal{C}_S$ is not $(\rho, \ell, \ell(1 + \alpha))$-list recoverable, there is a pair $S',\mathcal{C}'$ such that $S'$ is large and $|T(\mathcal{C'}) \cap S'|$ is unusually large. Taking a union bound over all candidates for $\mathcal{C}'$ then shows that there cannot be too many pairs of this sort. 
	
	Let $S \in \binom{[n]}{m}$ so that $\mathcal{C}_{S}$ is not $(\rho, \ell, \ell(1 + \alpha))$-list recoverable.
	We will show that there is a set $\mathcal{C}' \subset \mathcal{C}_S$ such that
	\begin{align}
	    \label{unbalanced:eq:Cbound} |\mathcal{C}'| &\leq 10 \sqrt{\ell/ \gamma}, \text{ and} \\
	    \label{unbalanced:eq:Tbound} |T(\mathcal{C}') \cap S| &\geq \sigma m/4.
	\end{align}
	
	Since $\mathcal{C}_S$ is not $(\rho, \ell, \ell(1 + \alpha))$-list recoverable, there are subsets $A_i \subseteq [q]$ for each $i \in S$ such that each $|A_i| \leq \ell$ and 
	$|\{c \in \mathcal{C}_S : \Delta(c, \prod_{i \in S} A_i) \leq \rho n \}| > \ell(1 + \alpha)$.
	
	Let 
	\[\mathcal{D} = \{c \in \mathcal{C}_s : \Delta (c, \prod_{i \in S} A_i  ) \leq \rho n \}.\]
	For $i \in S$, let
	\[\mathcal{D}_i = \{c \in \mathcal{D} : c[i] \in A_i\}. \]
	Let
	\[I = \{(c,i) \in \mathcal{D} \times S : c \in \mathcal{D}_i \}. \]
	From the definition of $\mathcal{D}$, we have
	\begin{equation}\label{eq:lowerBoundI} |I| \geq |\mathcal{D}| (1-\rho) m. \end{equation}
	
	Note that the average cardinality of the $\mathcal{D}_i$ is $(1-\rho) |\mathcal{D}|$. Let
	\[S' = \{i \in S : |\mathcal{D}_i| \geq (1-\rho)^2 |\mathcal{D}|\}. \]
	
	If $\rho = 0$, then $\mathcal{D}_i = \mathcal{D}$ for each $i$, and hence $|S'| = m$.
	Next we show that, if $\rho > 0$, then $|S'| \geq (1-\rho)(2-\rho)^{-1} m = \sigma m$.
    Since $|\mathcal{D}_i \leq |\mathcal{D}|$ for each $i$, we have
    \begin{equation}\label{eq:lowerBoundSPrime} |S'| \, |\mathcal{D}| \geq \sum_{i \in S'} |D_i| = |I| - \sum_{i \in S \setminus S'} \mathcal{D}_i. \end{equation}
    Since $|D_i| < (1- \rho)^2|\mathcal{D}|$ for each $i \in S \setminus S'$, we have
    \begin{equation}\label{eq:upperBoundSPrimeComplement} \sum_{i \in S \setminus S'} \leq (m - |S'|)(1-\rho)^2|\mathcal{D}|. \end{equation}
    
    A straightforward rearrangement of (\ref{eq:lowerBoundI}), (\ref{eq:lowerBoundSPrime}), and (\ref{eq:upperBoundSPrimeComplement}) using the assumption that $\rho > 0$ leads to the claimed lower bound on $|S'|$:
    \begin{equation}\label{eq:claimedBoundOnSPrime} |S'| \geq \sigma m. \end{equation}
    Since $\sigma < 1$, the bound $|S'| \geq \sigma m$ holds for the case $\rho = 0$ as well.
	
	For each $i \in S'$, choose a set $P_i \subset \binom{\mathcal{D}}{2}$ of $|P_i| \geq \gamma \ell/2$ disjoint pairs of codewords in $\mathcal{D}_i$ such that for each $\{c_1,c_2\} \in \mathcal{P}_i$, we have $c_1[i] = c_2[i]$. This is always possible since $|A_i| \leq \ell$ and $|\mathcal{D}_i| \geq (1+\rho)^2 |\mathcal{D}| \geq (1+\gamma)\ell$. Now choose $\mathcal{C'}$ randomly by including each element of $\mathcal{D}$ with probability $p = (\gamma \ell / 2)^{-1/2} \ell(1+\alpha) |\mathcal{D}|^{-1}$.
	Since $\ell \geq 4 \gamma^{-1}$ by hypothesis and $|\mathcal{D}| \geq \ell(1+\alpha)$ by the assumption that $\mathcal{C}_S$ is not $(\rho, \ell, \ell(1 + \alpha))$-list recoverable, we have $p < 1$.
	The expected size of $\mathcal{C}'$ is 
	\begin{equation*}
	\mathbf{E}[|\mathcal{C}'|] = p|\mathcal{D}| \leq (\gamma/ (2\ell))^{-1/2} (1 + \alpha) \leq (8\ell/\gamma)^{1/2}.
	\end{equation*}
	
	We remark that this is the only place where we use the assumption that $\alpha < 1$. For any fixed pair $c_1 \neq c_2$ of codewords in $\mathcal{D}$, the probability that both are included in $\mathcal{C}'$ is $p^2$.
	Since the pairs in $P_i$ are disjoint, the events that two distinct pairs $\{c_1, c_2\}, \{c_3, c_4\} \in P_i$ are both included in $\mathcal{C}'$ are independent.
	Hence, the probability that no pair in $P_i$ is included in $\mathcal{C}'$ is $(1-p^2)^{|P_i|} < e^{-p^2|P_i|} < 1/2$.
	Consequently, for each fixed $i \in S'$, the probability that $i \in T(\mathcal{C}')$ is greater than $1/2$.
	By linearity of expectation, $\mathbf{E}[|T(\mathcal{C}') \cap S'|] \geq |S'|/2 \geq \sigma m/2$.
	
	Let 
	\[ Y = |T(\mathcal{C}') \cap S'| - \frac{\sigma m}{4}\frac{|\mathcal{C}'|}{\mathbf{E}[|\mathcal{C}'|]}.\]
	By linearity of expectation, $\mathbf{E}[Y] \geq \sigma m/4$, hence there is some specific choice of $\mathcal{C}'$ for which $Y \geq \sigma m/4$.
	This can hold only if $|T(\mathcal{C}') \cap S| \geq |T(\mathcal{C}') \cap S'| \geq m/4$ and $|\mathcal{C}'| \leq 3 \mathbf{E}(|\mathcal{C}'|)$ simultaneously, which establishes (\ref{unbalanced:eq:Cbound}) and (\ref{unbalanced:eq:Tbound}). 
	
	Next we bound the probability that, for a fixed choice of $\mathcal{C}'$ and random $S$, we have have $|T(C') \cap S|$ large.
    Let $\mathcal{C}' \subset \mathcal{C}$ be an arbitrary set of $|\mathcal{C}'| \leq 10 \ell^{1/2} \gamma^{-1/2}$ codewords.
    Since the distance of $\mathcal{C}'$ is at least $n-nq^{-1} - d$ and $d \geq nq^{-1}$, we have
    \begin{equation}
    \label{unbalanced:eqn:exp}
    |T(\mathcal{C}')| \leq (nq^{-1} + d)\binom{|\mathcal{C}'|}{2} < d |\mathcal{C}'|^2. 
    \end{equation}
    
    For $S \in \binom{[n]}{m}$ chosen uniformly at random, $|T(\mathcal{C}') \cap S|$ follows a hypergeometric distribution.
    Specifically, we are making $m$ draws from a population size of $n$ of which $|T(\mathcal{C}')| \leq d |\mathcal{C}'|^2$ contribute to $|T(\mathcal{C}') \cap S|$.
    Using the assumption that $\ell \leq \gamma \sigma n (800 d)^{-1}$, the expected value of $|T(\mathcal{C}') \cap S|$ is
    \begin{equation}\label{eqn:expectationTCcapS}
        \mathbf{E}\left [|T(\mathcal{C}') \cap S| \right] \leq d |\mathcal{C}'|^2n^{-1}m \leq 100 \frac{d \ell} {\gamma n} m \leq \frac{\sigma m}{8}.  
    \end{equation}
    Next we use the following large deviation inequality for hypergeometric random variables (see~\cite{DP09}). Let $X$ be a hypergeometric random variable with mean $\mu$. Then for any $\alpha \geq 1$, 
    \begin{equation}
    \label{eqn:chernoff:hyp}
    \mathbb{P}(X \geq (1 + \alpha) \mu) \leq \exp(-\alpha \mu/4).
    \end{equation}
    Together with (\ref{eqn:expectationTCcapS}), this gives
    \begin{equation}
        \mathbb{P} (|T(\mathcal{C}') \cap S| \geq \sigma m/4) \leq  \exp \left (-\frac{\sigma m}{32} \right ).
    \end{equation}
    
    Finally, we take a union over all candidates for $\mathcal{C}'$.
    Let $X$ be the event that $\mathcal{C}_S$ is not $(\ell,\alpha,\rho)$ list recoverable, with $S \in \binom{[n]}{m}$ uniformly at random.
    Using the assumption that $\sigma m \geq 1280 \sqrt{\ell/\gamma} \log |\mathcal{C}|$, we have
    \begin{align*}
        \mathbb{P}(X) &\leq \sum_{\mathcal{C}' \subset \mathcal{C}_S: |\mathcal{C'}| \leq 10 \sqrt{\ell/ \gamma}} \mathbb{P}(|T(\mathcal{C}' \cap S)| \geq \sigma m/4) \\
        &\leq \binom{|\mathcal{C}|}{\lceil 10 \sqrt{\ell/\gamma} \rceil + 1} \exp \left (-\frac{m}{32} \right ) \\
        &< \exp \left (20 \sqrt{\ell/\gamma} \log |\mathcal{C}|  - \sigma m/32 \right ) \\
        &\leq \exp (-\sigma m/64),
    \end{align*}
    as claimed.
\end{proof}

We now show how to derive Theorem \ref{unbalanced:thm:simple} from Theorem \ref{unbalanced:thm:main}.

\begin{proof}[Proof of Theorem \ref{unbalanced:thm:simple}]
Suppose we have $\alpha, \rho, n, q,$ and $\epsilon$ as in the hypotheses of Theorem \ref{unbalanced:thm:simple}.
Let $\gamma = (1+\alpha)(1-\rho)^2 - 1$, $\sigma = (1-\rho)(2-\rho)^{-1}$ and $m = \lceil 1280 \epsilon^{-1} \log |C| \rceil$.
The singleton bound combined with the assumption that $\epsilon < c$ for a suitably chosen absolute constant $c$ implies that $m < n$.
Choose $S \in \binom{[n]}{m}$ uniformly at random.
The rate of $\mathcal{C}_S$ is
\[R= \log |\mathcal{C}| (m \log{q})^{-1} = \Omega(\epsilon (\log q)^{-1}).\]
It is straightforward to check that the hypotheses of Theorem \ref{unbalanced:thm:main} are satisfied if we take $\ell=\epsilon^{-2}\sigma^2\gamma$, and hence we have that $\mathcal{C}_S$ is $(\rho, \ell, \ell(1 +\alpha))$-list recoverable with high probability.
\end{proof}

\section{Upper bound}

Here we show the aforementioned upper bound for the rate to which a degree-$d$ Reed-Solomon code over $\mathbb{F}_q$ can be randomly punctured to be $(Omega(q), 1/2)$-zero-error list-recoverable. 

First, we recall a bit of standard and relevant sumset notation. For a group $G$ and subsets $A,B \subseteq G$, we denote the sumset $A+B = \{a + b ~|~ a\in A,~b \in B\}$. Clearly, we have $|A+B| \leq |A|\cdot |B|$. If $G = \mathbb{Z}_p$, then for $n < p/2$, we have that $[n] + [n] = \{2,\ldots,2n\}$. We are now ready to state and prove the upper bound.

\begin{theorem}
\label{unbalanced:thm:upperbound}
Let $m = o(\log q)$, and $S$ be a uniformly random subset of $\mathbb{F}_q$ of size $m$ where $q$ is a large prime. Then for every $d \geq 1$, the degree-$d$ Reed-Solomon code with the evaluation set at $S$ is, with high probability, not $(\Omega(q), 1/2)$-zero-error list-recoverable.
\end{theorem}

\begin{proof}
Let $S = \{s_0,\ldots, s_{m}\}$. Let $t$ be a large number such that $t^m =o (\sqrt{q})$. We are using the fact that $m = o(\log q)$ for the existence of such a $t$. W.L.O.G assume $s_0 = 0$ and $s_1 = 1$ (if $0,1 \not \in S$, then adding them to $S$ only makes the lower bound stronger). Consider the two sets 
\[
A_0 = \frac{1}{1 - s_2}[t] + \cdots \frac{1}{1 - s_{m-1}}[t]
\]
and
\[
A_1 = \frac{1}{s_2}[t] + \cdots \frac{1}{s_{m-1}}[t].
\]

\begin{claim}
With high probability over the choice of $S$, we have that $|A_0|,|A_1| = \Omega(\left(t^{m - 2}\right))$.
\end{claim}

\begin{proof}
We do the proof for $A_0$, the case for $A_1$ follows analogously. Let $P$ be the set of ``collisions'' in $A_0$. Formally:
\[
P := \left\{\left(a_2,\ldots,a_{m-2}, b_2,\ldots,b_{m-2}\right)~|~ \sum_{i = 2}^{m-2}a_is_i = \sum_{i = 2}^{m-2}b_is_i\right\}.
\]
So the number of distinct elements in $A_0$ is at least $t^{m - 2} - |P|$. We observe that 
\begin{align*}
\mathbf{E}[|P|] & = \sum_{\substack{{a_2,\ldots,a_{m -2}\in [t]} \\{b_{2},\ldots ,b_{m-2} \in [t]}}}\mathbb{P}\left(\sum_{i =2}^{m-2}a_is_i = \sum_{i =2}^{m-2}b_is_i\right) \\
& \leq \frac{1}{p}t^{2m - 4} \\
& = o(t^{m-2}).
\end{align*}
So by Markov's Inequality, with high probability, $|A_0| \sim t^{m - 2}$.
\end{proof}

Consider $\mathcal{D}$, the set of degree-$1$ Reed-Solomon codes given by the lines 
\[
\{Y = aX + b\}_{b \in A_0,a \in A_1}.
\]

First, we note that $|Y| = \Omega(t^{2m - 4}) $. Geometrically, $\mathcal{D}$ is just the set of all lines passing through some point of $\{0 \} \times A_0$ and $\{1\}\times A_1$. Clearly, $\{c[0]~|~ c \in \mathcal{C}\} = A_0$ and $\{c[1]~|~ c \in \mathcal{D}\} = A_1$. For $i \neq 0,1$, let us similarly define $A_i :=\{c[s_i]~|~c \in \mathcal{D}\}$.We have that 
\begin{align*}
A_i & = \{a(1 - s_i) + bs_i\}_{b \in A_0,a \in A_1} \\
& = (1 - s_i)\left(\frac{1}{1 - s_2}[t] + \cdots \frac{1}{1 - s_{m-1}}[t]\right) + s_i \left( \frac{1}{s_2}[t] + \cdots \frac{1}{s_{m-1}}[t]\right) \\
&= \left([t] +  \sum_{2 \leq j \leq m,~j \neq i}\frac{1 - s_i}{1 - s_j}[t]\right) + \left([t] + \sum_{2 \leq j \leq m,~j \neq i}\frac{s_i}{s_j}[t] \right) \\
& = \{2,\ldots, 2t\} + \sum_{2 \leq j \leq m,~j \neq i}\frac{1 - s_i}{1 - s_j}[t]+ \sum_{2 \leq j \leq m,~j \neq i}\frac{s_i}{s_j}[t].
\end{align*}

Thus, $|A_i| \leq (2t) \times t^{2m - 6} \leq 2t^{2m-5}$. 

This shows that there are lists $A_0, A_1,\ldots, A_m$ each of size at most $\ell := 2t^{2m - 5}$ such that there are at least  $\Omega(t^{2m - 4}) = \ell^{1 + \frac{1}{2m}}$ codewords, namely $\mathcal{D}$, contained in $A_0 \times \cdots \times A_m$.
\end{proof}

For a fixed $d$, the above theorem rules out hope of randomly puncturing degree-$d$ Reed-Solomon codes to rate $\omega\left(\frac{1}{\log q}\right)$ for the desired list recoverability. We believe that this is essentially the barrier. We state the concrete conjecture that we alluded to in Section~\ref{unbalanced:sec:motivation}.

\begin{conjecture}
\label{unbalanced:conj:punc}
The degree-$d$ Reed-Solomon code with evaluation set $\mathbb{F}_q$ can be randomly punctured to rate $\Omega_d\left(\frac{1}{\log q}\right)$ so that is it $(\Omega(q), 1/2)$-list recoverable with high probability.
\end{conjecture}

\section{Open problems}

The main open problem that we would like to showcase is Conjecture~\ref{unbalanced:conj:punc}. This was probably believed to be true but we could not find it written down explicitly in the literature. List recoverable codes have connections to various other combinatorial objects (see~\cite{VAD07}) and if true, Conjecture~\ref{unbalanced:conj:punc} could lead to the construction of some other interesting combinatorial objects. 

The second open problem is to derandomize Theorem~\ref{unbalanced:thm:simple}, i.e., to find an \emph{explicit} Reed-Solomon code which is list recoverable beyond the Johnson bound at least in the zero-error case. Understanding how these evaluation sets look like could lead to progress on Conjecture~\ref{unbalanced:conj:punc}, or could be interesting in its own right. 

Finally, the last open problem is that given a Reed-Solomon code $\mathcal{C} \subset [q]^m$ of rate $R$ on a randomly chosen evaluation set $S$, find an efficient algorithm for list recovery, i.e., take input lists $A_1,\ldots, A_{m}$ of size $O(R^{-2}(\log q)^{-1})$, and output all the codewords contained in $A_1\times \cdots \times A_{m}$ with high probability (over the choice of $S$ and the randomness used by the algorithm). This would also likely require some understanding of the properties of the evaluation set. 
\bibliography{references.bib}
\bibliographystyle{plain}

\end{document}